\documentclass[aap]{amsart}
\input{mystyle.sty}

\usepackage[top=1.25in, bottom=1.25in, left=0.8in, right=0.8in]{geometry}
\usepackage{amssymb}
\usepackage{amsthm}

\def\OL{\mathbb{O}} 
\def\OLb{\OL_b} 
\def\AM{\mathbb{M}} 

\title[]{Mimicking and Conditional Control with Hard Killing}
\author{Rene Carmonar}
\address{Department of Operations Research and Financial Engineering, Princeton University, Princeton NJ 08544}
\email{rcarmona@princeton.ed}

\author{Daniel Lacker}
\address{Department of Industrial Engineering and Operations Research, Columbia University}
\email{dlacker@gmail.com}

\thanks{R.C. was partially supported by AFOSR grant FA9550-23-1-0324. D.L. acknowledges support from an Alfred P. Sloan Fellowship and the NSF CAREER award DMS-204532}
\address{}

\date{}

\begin{document}

\begin{abstract}
We first prove a mimicking theorem (also known as a Markovian projection theorem) for the marginal distributions of an It\^o process conditioned to not have exited a given domain. We then apply this new result to the proof of a conjecture of P.L. Lions for the optimal control of conditioned processes.
\end{abstract}

\maketitle


\section{\textbf{Introduction}}

The motivation of this paper is to answer a question raised by P.L. Lions in his lectures at the Coll\`ege de France in November 2016. See \cite{Lions2016CDF}. There, he proposed to study the optimal control of an It\^o process when the objective function to minimize is computed as the integral over time of conditional expectations of a cost function, the conditional expectation at time $t$ requiring that the process did not exit a given domain before that time. As originally stated, the problem does not fit in the usual categories of stochastic control problems considered in the literature, so its solution requires new ideas, if not new technology. In his lectures, Lions emphasized the strong dependence of the running cost upon the past, and in so doing, raised the question of the possible differences between the value functions resulting from optimization over the class of Markovian controls as opposed to the general family of open loop controls merely assumed to be adapted. The equality of these value functions is expected to hold in the classical theory of stochastic control, as it can be proved under quite general conditions. However, in the present context, the values of the objective function depend strongly upon the past history of the controlled trajectories, raising the specter of possible differences between the results of the optimization over these two classes of control processes. The goal of this paper is to give a short proof that these two values are the same under the conditions originally chosen by P.L. Lions. Our approach is to prove a specific form of the mimicking theorem for killed processes, which we then use to argue that any open-loop control can be replaced by a Markovian one while improving the cost. This idea of using the mimicking theorem to  prove equality of the open-loop and Markovian values seems to have first appeared in \cite{lacker2015mean}, with earlier work such as \cite{nicole1987compactification} using Krylov's selection argument.It was a crucial component of the analysis of the soft killing case in 
\cite{CarmonaLauriereLions}.

In no small part due to the technicality of the original presentation in \cite{Lions2016CDF}, very few papers have attempted to study this class of control problems. The only works we know of are \cite{AchdouLauriereLions}
which treats numerically the case of Markovian feedback controls for large time horizon and large domains,
and \cite{CarmonaLauriereLions} where the problem is completely analyzed in the case of soft killing, as opposed to the hard killing originally suggested in \cite{Lions2016CDF} and treated in this paper. The paper \cite{NutzZhang} studies a related problem of optimal stopping.

\section{\textbf{The Original Conditional Exit Control Problem}}
\label{sec:conditional_exit}

We recall the formulation of the conditional control problem originally proposed by P.L. Lions in his lectures at the Coll\`ege de France in 2016. See \cite{CarmonaLauriereLions} for details and a complete analysis of the soft killing version of the model. We here set up the control problem in both open-loop and Markovian forms.

Throughout the paper, we assume that  $D$ is a bounded nonempty open domain in $\RR^d$ with a smooth boundary $\partial D$. 
We are given a time horizon $T > 0$ and two continuous cost functions, $f : \overline{D} \times \RR^d \to \RR$ and $g: \overline{D} \to \RR$. We assume that $a \mapsto f(x,a)$ is convex for each $x$, and that $f$ and $g$ satisfy the growth conditions
\begin{equation}
|f(x,a)| + |g(x)| \le c(1+|x|^2 + |a|^2), \label{asmp:growth}
\end{equation}
for some constant $c$. The particular case $f(x,a)=\frac12|a|^2 + \tilde f(x)$ for a bounded continuous function $\tilde f$ was consider in the original formulation of the problem in \cite{Lions2016CDF}.
We work with a filtered probability space $(\Omega,\cF,\FF=(\cF_t)_{t \in [0,T]},\PP)$ supporting a $d$-dimensional $\FF$-Brownian motion $\bW=(W_t)_{t \in [0,T]}$ and an $\cF_0$-measurable random variable $\xi$ taking values in $D$.

To fix notation, we denote by $C([0,\infty);\RR^d)$ the space of continuous functions of time $t\in[0,\infty)$ with values in $\RR^d$, and by $C_0([0,\infty);\RR^d)$ the subspace of those functions $x\in C([0,\infty);\RR^d)$ satisfying $x(0)=0$. If $x\in C([0,\infty);\RR^d)$, we denote by $\tau(x)$ the first exit time of the path $x$ from $D$, namely the quantity:
\begin{equation*}
    \tau(x)=\inf\{t\ge 0;\,x(t)\notin D\}
\end{equation*}
with the convention that $\inf \emptyset = \infty$. 
Also, whenever $X$ is a random variable, we denote by $\cL(X)$ its distribution and by $\cL(X|A)$ its conditional distribution given the event $A$.

The open-loop problem is defined as follows. 
Let $\OL$ denote the set of progressively measurable processes $\balpha=(\alpha_t)_{0\le t\le T}$  satisfying
\begin{equation*}
\EE\int_0^T|\alpha_t|^2dt<\infty.
\end{equation*}
This is the set of \emph{open-loop controls}. Let $\OLb$ denote the subset of controls which are a.e.\ a.s.\ bounded.
For $\balpha \in \OL$ we define the corresponding controlled state process $\bX=(X_t)_{t \in [0,T]}$ by
\begin{equation*}
dX_t=\alpha_t dt + \sigma dW_t, \quad X_0=\xi.
\end{equation*}
The corresponding cost $J(\balpha)$ is then defined by
\begin{equation*}
\begin{split}
J(\balpha)&=\int_0^T\EE[f(X_t,\alpha_t)|\tau(X)> t]\;dt +\EE[g(X_T)|\tau(X)> T] \\
&=\int_0^T
\frac{\EE\Bigl[f(X_t,\alpha_t)\textbf{1}_{\{\tau(X)>t\}}\Bigr]}{\PP[\tau(X)> t]}dt
+\frac{\EE\Bigl[g(X_T)\textbf{1}_{\{\tau(X)>T\}}\Bigr]}{\PP[\tau(X)> T]}.
\end{split}
\end{equation*}
If $\PP(\tau(X)> T)=0$, we adopt the convention that $J(\alpha)=\infty$.

The Markovian problem  is defined as follows.
Let $\AM$ denote the set of Borel measurable functions $\phi : [0,T] \times \RR^d\to \RR^d$ such that the SDE
 \begin{equation*}
     dX_t = 1_D(X_t)\phi(t,X_t)dt + 1_D(X_t)dW_t, \quad X_0=\xi,
 \end{equation*}
 admits a strong solution satisfying
 \begin{equation*}
     \EE\int_0^T |\phi(t,X_t)|^2\,dt < \infty.
 \end{equation*}
 This is the set of \emph{Markovian controls}. The corresponding cost is defined by
 \begin{equation*} 
     J(\phi) =\int_0^T\EE[f(X_t,\phi(t,X_t))|\tau(X)> t]\;dt +\EE[g(X_T)|\tau(X)> T].
 \end{equation*}
 We again adopt the convention that $J(\phi)=\infty$ if $\PP(\tau(X)> T)=0$.
 
\vskip 6pt
Our main result is the following:

\begin{theorem} \label{th:main}
The value over open-loop and Markovian controls is the same: 
\begin{equation*}
     \inf_{\alpha \in \AM}J(\phi) = \inf_{\alpha \in \OL}J(\balpha).
 \end{equation*}
\end{theorem}
  
The proof is divided between the next two sections. First, Section \ref{se:mimicking} proves a mimicking theorem, which may be of interest in its own right, and Section \ref{se:mainproof} then uses it to prove Theorem  \ref{th:main}.

\section{\textbf{A Mimicking Theorem for Conditional Marginals}} \label{se:mimicking}

The goal of this section is to prove a conditional form of the classical mimicking theorem for killed It\^o processes. Conditioning at each time $t$ on the fact that the process did not exit the domain $D$ up to and including time $t$ is not covered by the classical versions of the mimicking theorem as introduced by Gy\"ongy in \cite{gyongy1986mimicking}  and generalized by Brunick and Shreve in \cite{BrunickShreve}, and it requires a special treatment.
While our proof below applies to more general It\^o processes having non-constant volatility, we restrict our presentation to the case of the volatility equal to the identity for the sake of ease of exposition. We will need the following lemma whose proof we could not find in the existing literature.

\begin{lemma}
    \label{le:well_posedness}
If $b:[0,\infty)\times\RR^d\ni (t,x)\mapsto b(t,x)\in\RR^d$ is a bounded measurable function, the stochastic differential equation (SDE)
\begin{equation}
    \label{fo:sde}
dX_t=\textbf{1}_D(X_t)b(t,X_t)dt+\textbf{1}_D(X_t)dW_t
\end{equation}
admits a pathwise unique strong solution started from any initial condition.
\end{lemma}
\begin{proof}
We first prove existence of a strong solution. By a well known result of Veretennikov \cite[Theorem 4]{veretennikov1980strong}, the SDE
\begin{align*}
dY_t = b(t,Y_t)dt + dW_t
\end{align*}
admits a pathwise unique strong solution for every initial condition. Setting $X_t = Y_{t \wedge \tau(Y)}$ produces a strong solution of the SDE \eqref{fo:sde}.

It remains to show that uniqueness in law holds for the SDE \eqref{fo:sde};  indeed, according to \cite[Theorem 3.2]{cherny2002uniqueness}, uniqueness in law plus strong existence together imply pathwise  uniqueness.
Using Theorem 6.4.3 (or 6.4.2) of \cite{StroockVaradhan} reduces our task to the proof of the uniqueness in law of the equation
\begin{equation}
    \label{fo:reduced_sde}
dX_t=\textbf{1}_D(X_t)dW_t.
\end{equation}
Let us assume that $\bX=(X_t)_{t\ge 0}$ is a weak solution of \eqref{fo:reduced_sde}.
Clearly, we can limit ourselves to initial conditions $X_0=x_0$ in $D$.  
Using the multivariate version of the Dambis-Dubins-Schwarz theorem,  see for example \cite[Theorem 18.4]{Kallenberg}, if we set $M_t=X_t-x_0$, then
$\bM=(M_t)_{t\ge 0}$ is an isotropic continuous local martingale with $M_0=0$ (isotropic means that $[M^i]=[M^j]$ and $[M^i,M^j]=0$ for all $i\ne j$), so if we define
\begin{equation}
    \label{fo:time_change}
\tau_s=\inf\{t\ge 0;\,A_t>s\} 
\quad\text{with}\quad 
A_t=\int_0^t\textbf{1}_D(X_s)ds=[M^i]_t
\quad\text{with}\quad
\cG_s=\cF_{\tau_s}
\end{equation}
then there exists a standard extension $\tilde{\cG}$ of the filtration $\cG$ and a $\tilde{\cG}$-Brownian motion $\bZ=(Z_t)_{t\ge 0}$ such that $Z=M\circ\tau$ a.s. on $[0,A_\infty)$ and $M=Z\circ A$.

\vskip 2pt
Let us denote by $\rho(x)=d(x,D)=\inf_{y\in D}\|x-y\|$ the distance from $x$ to the domain $D$, and let $\epsilon >0$ be fixed momentarily.
We define:
$$
\theta_1=\inf\{t>0;\;\rho(X_t)>\epsilon\}
\qquad\text{and}\qquad
\theta_2=\inf\{t>\theta_1;\; \rho(X_t)>2\epsilon \quad\text{or}\quad\rho(X_t)=0\}.
$$
For each finite $T>0$,
$$
\EE\Bigl[\int_{\theta_1\wedge T}^{\theta_2\wedge T}\textbf{1}_D(X_s)\; ds \Bigr] = 0.
$$
since either $\theta_1>T$ in which case $\theta_1\wedge T=\theta_2\wedge T=T$, or $\theta_1\le T$, in which case $\theta_1\wedge T=\theta_1$ and $\textbf{1}_D(X_s)=0$ for all $s\in[\theta_1,\theta_2)$. As a result,
\begin{equation}
    \label{fo:zero}
\EE\Bigl[\|X_{\theta_1\wedge T}-X_{\theta_2\wedge T}\|^2\Bigr]=0.
\end{equation}
On the event $\{\sup_{t\ge 0}\rho(X_t)>2\epsilon\}$ we have
$$
\lim_{T\nearrow\infty}X_{\theta_1\wedge T}= X_{\theta_1}
\qquad\text{and}\qquad
\lim_{T\nearrow\infty}X_{\theta_2\wedge T}= X_{\theta_2}
$$
because both $\theta_1$ and $\theta_2$ are finite on this event. So by Lebesgue's dominated convergence theorem, we have
$$
\EE\Bigl[\textbf{1}_{\{\sup_{t\ge 0}\rho(X_t)>2\epsilon\}}\|X_{\theta_1}-X_{\theta_2}\|^2\Bigr]=\lim_{T\nearrow\infty}\EE\Bigl[\textbf{1}_{\{\sup_{t\ge 0}\rho(X_t)>2\epsilon\}}\|X_{\theta_1\wedge T}-X_{\theta_2\wedge T}\|^2\Bigr]
$$
and this limit is $0$ because of \eqref{fo:zero}. Since $\rho(X_{\theta_1})=\epsilon$ and  $\rho(X_{\theta_2})=2\epsilon$ or $\rho(X_{\theta_2})=0$ on the event $\{\sup_{t\ge 0}\rho(X_t)>2\epsilon\}$, $\|X_{\theta_1}-X_{\theta_2}\|^2>0$ on that event, and we conclude that 
$$
\PP\Bigl[\sup_{t\ge 0}\rho(X_t)\le 2\epsilon\Bigr]=1.
$$
Since $\epsilon>0$ was arbitrary, we conclude that almost surely, $X_t\in \bar D$ for all $t\ge 0$.
Burkh\"older-Davis-Gundy inequality implies that $A_\infty<\infty$ a.s. and since $M=Z\circ A$, this implies that $M_t$ and our weak solution $X_t$ stop moving at $t=A_\infty$. As a result $X_{A_\infty}\notin  D$.
In fact, we claim that $A_\infty=\tau(X)$ the first exit time from $D$. Indeed, if $A_\infty>\tau(X)$, $X_t$ has to bounce back into $D$ instantaneously after $t=\tau(X)$ because, not only $X_t$ cannot go into $\bar D^c$, but it cannot stay on $\partial D$ for a non-empty time interval because the intervals of consistency of $X_t$ and $A_t$ have to be the same. This last fact implies that the time change $s\mapsto \tau_s$ defined in \eqref{fo:time_change} is continuous. Since
$$
\{A_\infty>\tau(X)\}\subset\{\omega\in \Omega;\;\omega_0=0,\;\exists t_0,\;\exists \epsilon>0, \;x_0+\omega_{t_0}\in\partial D,\;\text{and}\; x_0+\omega_t\in D\;\text{for}\;t\in(t_0-\epsilon,t_0)\cup(t_0,t_0+\epsilon)\},
$$
and since $\partial D$ is smooth, the Wiener measure of the set of paths in the above right hand side is $0$, the fact that $Z=M\circ\tau$ is a $\tilde{\cG}$-Brownian motion, implies that $\PP[A_\infty>\tau(X)]=0$. This identifies completely the law of $\bX$ and completes the proof of the uniqueness in law.
\end{proof}

We now state and prove the desired version of the mimicking theorem.

\begin{proposition}
    \label{pr:mimicking}
Let $x_0\in D$ and let us assume that $\bX=(X_t)_{0\le t\le T}$ is an It\^o process of the form 
\begin{equation}
    \label{fo:Xoft}
X_t=x_0+\int_0^t\alpha_s\;ds+W_t
\end{equation}
where $\bW=(W_t)_{0\le t\le T}$ is a Wiener process and $\balpha=(\alpha_t)_{0\le t\le T}$ is a bounded progressively measurable process. Then there exists a (deterministic) bounded measurable function $\tilde\alpha:[0,T]\times\RR^d\mapsto \RR^d$ and a strong solution  $\tilde{\bX}=(\tilde X_t)_{0\le t\le T}$ of the SDE
\begin{equation}
    \label{fo:Xtildeoft}
\tilde X_t=x_0+\int_0^t\tilde\alpha(s,\tilde X_s)\;ds+\tilde W_t
\end{equation}
such that $\cL(X_t\,|\,\tau(X)>t)=\cL(\tilde X_t\,|\,\tau(\tilde   X)>t)$ for all $t\in[0,T]$. In fact we may choose:
\begin{equation}
    \label{fo:alphatilde}
\tilde\alpha(t,x)={\bf 1}_D(x)\EE[\alpha_t\;|\;X_{t\wedge \tau(X)}=x].
\end{equation}
\end{proposition}

\begin{proof}
    Let $Y_t=X_{t\wedge \tau(X)}$ and note that 
$$
dY_t={\bf 1}_{\{\tau(X)>t\}}dX_t={\bf 1}_{\{\tau(X)>t\}}\alpha_t\,dt+{\bf 1}_{\{\tau(X)>t\}} dW_t
$$
Note also that $\{\tau(X)>t\}=\{Y_t\in D\}$ a.s., i.e. the events differ at most by a null set. Indeed, the boundary $\partial D$ being smooth implies that $X_{\tau(X)}\in \bar D\setminus D$. As a result, $\bY=(Y_t)_{0\le t\le T}$ is a solution of the SDE:
\begin{equation}
    \label{fo:dYt}
dY_t={\bf 1}_{D}(Y_t)\alpha_t\,dt+{\bf 1}_{D}(Y_t) dW_t
\end{equation}
with measurable bounded coefficients, to which we can apply the standard mimicking theorem. See \cite{BrunickShreve}. Notice that if $\tilde \alpha$ is defined by \eqref{fo:alphatilde} we have:
$$
\EE[{\bf 1}_D(Y_t)\alpha_t\;|\,Y_t=y]={\bf 1}_D(y)\EE[\alpha_t\;|\;X_{t\wedge\tau(X)}=y]=\tilde\alpha(t,y).
$$
and the classical mimicking theorem gives the existence of a weak solution $\tilde{\bY}=(\tilde Y_t)_{0\le t\le T}$ of the SDE
$$
d\tilde Y_t=\tilde\alpha(t,\tilde Y_t)dt+{\bf 1}_D(\tilde Y_t)d\tilde W_t
$$
such that $\cL(\tilde Y_t)=\cL(Y_t)$ for all $t\in[0,T]$. 
Note by Lemma \ref{le:well_posedness} that $\tilde Y$ is in fact a strong solution.
We claim that for every $t>0$,
\begin{equation}
    \label{fo:claim}
\{\tilde Y_t\in D\}=\{\tau(\tilde Y)>t\}\qquad \tilde \PP-a.s.
\end{equation}
To see this, note that $\tilde Y'_t = \tilde Y_{t \wedge \tilde(Y)}$ defines another strong solution of the same SDE because $\tilde \alpha(t,\cdot)=0$ on $D^c$. By pathwise uniqueness (Lemma \ref{le:well_posedness}) we must have $\tilde Y' \equiv \tilde Y$, which means that $\tilde Y$ stops moving as soon as it hits $\partial D$. The claim \eqref{fo:claim} follows.

Now, since $\tilde\alpha$ is bounded, the SDE
$$
d\tilde X_t = \tilde\alpha(t,\tilde X_t)dt+d\tilde W_t,\qquad \tilde X_0=x_0
$$
has a unique weak solution, and since \emph{gluing} $(\tilde Y_t)_{0\le t\le \tau(\tilde Y)}$ to $(\tilde X_t)_{\tau(\tilde X)\le t\le T}$ gives another solution of the martingale problem for $(I_d,\tilde\alpha)$, see for example the proof of \cite[Lemma 7.2.3]{StroockVaradhan} and \cite[Theorem 6.6.1]{StroockVaradhan}, we conclude by uniqueness of this martingale problem that the restrictions to $\cF_\tau$ of the laws of $\tilde{\bY}$ and $\tilde{\bX}$ coincide. In particular, for each $t\in [0,T]$, we have
\begin{equation}
    \label{fo:den}
\tilde{\PP}[\tau(\tilde X)>t]=\tilde{\PP}[\tau(\tilde Y)>t]=\tilde{\PP}[\tilde Y_t\in D]=\PP[Y_t\in D]=\PP[\tau(X)>t]
\end{equation}
where we used \eqref{fo:claim} and the mimicking property $\cL(Y_t)=\cL(\tilde Y_t)$.
More generally, for each bounded measurable function $\varphi:\RR^d\mapsto\RR$ we also have:
\begin{equation}
    \label{fo:num}
\begin{split}
\tilde{\EE}[\varphi(\tilde X_t){\bf 1}_{\{\tau(\tilde X)>t\}}]
&= \tilde{\EE}[\varphi(\tilde X_{t\wedge \tau(\tilde X)}){\bf 1}_{\{\tau(\tilde X)>t\}}]\\
&= \tilde{\EE}[\varphi(\tilde Y_{t\wedge \tau(\tilde Y)}){\bf 1}_{\{\tau(\tilde Y)>t\}}]\\
&= \tilde{\EE}[\varphi(\tilde Y_{t}){\bf 1}_{D}(\tilde Y_t)]\\
&= \EE[\varphi(Y_{t}){\bf 1}_{D}(Y_t)]\\
&=\EE[\varphi(X_t){\bf 1}_{\{\tau(X)>t\}}].
\end{split}
\end{equation}
Again, we used \eqref{fo:claim} and $\cL(Y_t)=\cL(\tilde Y_t)$.
Putting \eqref{fo:num} and \eqref{fo:den} gives the desired result.
\end{proof}

\begin{remark}
The proof of Proposition 3.6 in \cite{campi2021n} contains an argument  similar to our Proposition \ref{pr:mimicking}, though there is a subtle gap in their argument. Specifically, they appear to take the identity \eqref{fo:claim} for granted, and its careful justification is where we need to invoke Lemma \ref{le:well_posedness}.
\end{remark}

\section{\textbf{Proof of Theorem \ref{th:main}}}  \label{se:mainproof}

We next prove Theorem \ref{th:main}. Note first that any Markovian control trivially induces an open-loop control in $\OL$, by setting $\alpha_t(\omega) = \phi(t,X_t(\omega))$. Hence,
 \begin{equation*}
     \inf_{\balpha \in \OL}J(\balpha) \le \inf_{\phi \in \AM}J(\phi).
 \end{equation*}
We prove Theorem \ref{th:main} by establishing separately the following two claims:
\begin{align}  
 \inf_{\phi \in \AM}J(\bphi)  \le \inf_{\balpha \in \OLb}J(\balpha) \label{fo:equality} \\
\inf_{\balpha \in \OLb}J^\tau(\balpha) \le \inf_{\balpha \in \OL}J^\tau(\balpha). \label{fo:equality2}
\end{align}
From this it will follow that the three infima above are all equal.

\subsection*{Proof of \eqref{fo:equality}} 
Fix $\balpha=(\alpha_t)_{t\ge 0} \in \OL_b$, let us denote by $\bX=(X_t)_{0\le t\le T}$ the associated controlled process $\bX^\alpha$ over the interval $[0,T]$.
\begin{equation}
\begin{split}
J(\balpha)
&=\int_0^T
\frac{\EE^{\PP}\Bigl[f(X_{t\wedge \tau(X)},\alpha_t)\textbf{1}_{\{\tau(X)>t\}}\Bigr]}{\EE^{\PP}[\tau(X)>t]}dt +\EE^\PP\Bigl[g(X_T)\,|\,\tau(X)>T\Bigr] \\
&=\int_0^T
\frac{\EE^\PP\Bigl[\EE^\PP[f(X_{t\wedge\tau(X)},\alpha_t)|X_{t\wedge\tau(X)}]\textbf{1}_{\{\tau(X)>t\}}\Bigr]}{\EE^{\PP}[\tau(X)>t]}dt
+\EE^\PP\Bigl[g(X_T)\,|\,\tau(X)>T\Bigr].
\end{split}
\end{equation}
where we used the fact that $\textbf{1}_{\{\tau(X)>t\}}$ is measurable with respect to $X_{t\wedge\tau(X)}$ since $\{\tau(X)>t\}=\{X_{t\wedge\tau(X)}\in D\}$. By convexity of $a\mapsto f(x,a)$ and Jensen's inequality,
\begin{equation}
\begin{split}
J(\balpha)
&\ge\int_0^T
\frac{\EE^\PP\Bigl[f\Bigl(X_{t\wedge\tau(X)},\EE^\PP[\alpha_t|X_{t\wedge\tau(X)}]\Bigr)\textbf{1}_{\{\tau(X)>t\}}\Bigr]}{\PP[\tau(X)>t]}dt
+\EE^\PP\Bigl[g(X_T)\,|\,\tau(X)>T\Bigr].
\end{split}
\end{equation}
Letting $\tilde X_t$ be the process constructed in Proposition \ref{pr:mimicking}, with $\tilde \alpha$ the function defined therein, we have
\begin{equation}
\begin{split}
J(\balpha)
&\ge\int_0^T
\frac{\EE^\PP\Bigl[f\Bigl(X_{t\wedge\tau(X)},\tilde\alpha(t,X_{t\wedge\tau(X)})\Bigr)\textbf{1}_{\{\tau(X)>t\}}\Bigr]}{\PP[\tau(X)>t]}dt
+\EE^\PP\Bigl[g(X_T)\,|\,\tau(X)>T\Bigr]\\
&=\int_0^T
\frac{\EE^\PP\Bigl[f\Bigl(X_t,\tilde\alpha(t,X_t)\Bigr)\textbf{1}_{\{\tau(X)>t\}})\Bigr]}{\PP[\tau(X)>t]}dt
+\EE^\PP\Bigl[g(X_T)\,|\,\tau(X)>T\Bigr]\\
&=\int_0^T
\EE^\PP\Bigl[f(X_t,\tilde\alpha(t,X_t))\,|\,\tau(X)>t\Bigr]dt
+\EE^\PP\Bigl[g(X_T)\,|\,\tau(X)>T\Bigr]\\
&=\int_0^T
\EE^{\tilde \PP}\Bigl[f(\tilde X_t,\tilde\alpha(t,\tilde  X_t))\,|\,\tau(\tilde X)>t\Bigr]dt
+\EE^\PP\Bigl[g(\tilde X_T)\,|\,\tau(\tilde X)>T\Bigr]\\
\end{split}
\end{equation}
where the last step used the equality of laws given by Proposition \ref{pr:mimicking}. Noting that $\tilde \alpha \in \AM$, this proves \eqref{fo:equality}.

\subsection*{Proof of \eqref{fo:equality2}}

Fix $\balpha \in \OL$. For each $n \in \mathbb{N}$ define $\balpha^n=(\alpha^n_t)_{0\le t\le T} \in \OL_b$ by setting $\alpha^n_t$ to be the projection of $\alpha_t$ onto the centered ball of radius $n$ in $\RR^d$. Let $\bX^n$ be the corresponding state process. We claim that $J(\balpha^n) \to J(\balpha)$, which will complete the proof.
 To prove this, we let
 \begin{equation*}
     Y_t = X_0 + \int_0^t\alpha_s\,ds + W_t, \qquad Y_t^n = X_0 + \int_0^t\alpha_s^n\,ds + W_t, \quad t \in [0,T].
 \end{equation*}
 Note that $\tau(X)=\tau(Y) =: \tau$ and $X_t=Y_{t \wedge \tau}$, and similarly  $\tau(X^n)=\tau(Y^n) =: \tau^n$ and $X_t^n=Y^n_{t \wedge \tau^n}$. We clearly have
 \begin{equation*}
    \EE\big[\sup_{t \in [0,T]} |Y^n_t-Y_t|^2\big] \le T\EE\int_0^T|\alpha^n_t - \alpha_t|^2\,dt \to 0.
 \end{equation*}
By square-integrability of $\balpha$, it is well known \cite[Theorem 7.2]{liptser2013statistics} that the law of the process $\bY-X_0$ is absolutely continuous with respect to the Wiener measure, in notation $\mathcal{L}(\bY_\cdot-X_0) \ll \mathcal{L}(\bW_\cdot)$. By smoothness of the boundary of $D$,  Brownian motion cannot ``graze" $\partial D$. That is, with probability 1, there do not exist $t \in (0,T)$ and $\epsilon > 0$ such that $W_t \in \partial D$ and $W_s \in \overline{D}$ for $s \in (t ,t+\epsilon)$. Hence, $\bY$ satisfies the same property. From Lemma \ref{le:grazing} below, we deduce that $\tau^n = \tau(Y^n)\to \tau(Y)=\tau$ a.s. It follows that
\begin{align*}
\EE\Big[ \sup_{t \in [0,T]}|X^n_t-X_t|^2\Big] &= \EE\Big[ \sup_{t \in [0,T]}|Y^n_{t\wedge \tau^n}-Y_{t\wedge \tau}|^2\Big] \\
    &\le 2\EE\Big[ \sup_{t \in [0,T]}|Y^n_t-Y_t|^2\Big] + 2\EE\Big[ \sup_{t \in [0,T]}|Y_{t\wedge \tau^n}-Y_{t\wedge \tau}|^2\Big] \to 0.
\end{align*}
Using the quadratic growth assumption \eqref{asmp:growth},  we readily deduce that $J(\balpha^n) \to J(\balpha)$.

\begin{lemma} 
\label{le:grazing}
Let $x^n, x \in C([0,\infty);\RR^d)$ with $x^n \to x$ uniformly on compacts, and with $x_0,x^n_0 \in D$. Suppose that there do not exist $t >0$ and $\epsilon > 0$ such that $x_t \in \partial D$ and $x_s \in \overline{D}$ for $s \in (t,t+\epsilon)$. Then $\tau(x^n) \to \tau(x)$. 
\end{lemma}
 
\begin{proof}
First suppose $t < \tau(x)$. Then $x_s \in D$ for all $s \in [0,t]$. By uniform convergence and openness of $D$, for sufficiently large $n$ we have $x^n_s \in D$ for all $s \in [0,t]$, and thus $\tau(x^n) > t$. Hence, $\liminf_n \tau(x^n) \ge t$. This holds for every $t < \tau(x)$, which shows that $\liminf_n \tau(x^n) \ge \tau(x)$.

 Next, suppose $t > \tau(x)$. Let $\epsilon > 0$. By assumption, $\tau(x)$ must be an accumulation point of the set $\{s > \tau(x) : x_s \notin \overline{D}\}$, so we may find $s \in (\tau(x), \tau(x) +\epsilon)$ such that $x_s \notin \overline{D}$. Since $\overline{D}^c$ is open, we have for sufficiently large $n$ that $x^n_s \notin \overline{D}^c$, which implies $\tau(x^n) \le s < \tau(x) + \epsilon$. We deduce that $\limsup_n \tau(x^n) \le \tau(x)$.
 \end{proof}

\bibliographystyle{apalike}

\end{document}